\documentclass{article}
\usepackage{pgfplots}
\pgfplotsset{compat=1.18}
\usepackage{mystyle}
\usepackage[utf8]{inputenc}
\usepackage{hyperref}
\usepackage{cleveref}
\usepackage{amsmath}
\usepackage{amsfonts}
\usepackage{amssymb}
\usepackage{amsbsy}
\usepackage{amsthm}
\usepackage{algorithm}
\usepackage[noend]{algpseudocode}
\newtheorem{proposition}{Proposition}

\newtheorem{inequality}{Inequality}
\newtheorem{definition}{Definition}

\newtheorem{assumption}{Assumption}
\crefname{assumption}{Assumption}{Assumptions}
\Crefname{assumption}{Assumption}{Assumptions}

\def\Var{\text{Var}}

\title{Stochastic Adaptive Optimization with Unreliable Inputs:\\
A Unified Framework for High-Probability Complexity Analysis} 
\author{Katya Scheinberg\thanks{{School of Industrial and Systems Engineering, Georgia Tech, Atlanta, GA, USA; E-mail: katya.scheinberg@isye.gatech.edu}}
   \and Miaolan~Xie\thanks{{Edwardson School of Industrial Engineering, Purdue University, West Lafayette, IN, USA; E-mail: miaolanx@purdue.edu}}}

\date{}

\begin{document}

\maketitle

\begin{abstract}
	We consider an unconstrained continuous optimization problem where, in each iteration, gradient estimates may be arbitrarily corrupted with a probability greater than $\frac 1 2$. Additionally, function value estimates may exhibit heavy-tailed noise. This setting captures challenging scenarios where both gradient and function value estimates can be unreliable, making it applicable to many real-world problems, which can have outliers and data anomalies. 
	We introduce an algorithmic and analytical framework that provides high-probability bounds on iteration complexity for this setting. The analysis offers a unified approach, encompassing methods such as line search and trust region.
\end{abstract}

\section{Introduction}

In this paper, we are interested in finding an $\epsilon$-stationary point of an unconstrained, differentiable, possibly non-convex function $$\phi: \R^n \to \R.$$ 
We make the standard assumption that $\nabla \phi$ is $L$-Lipschitz and bounded below by some constant $\phi^*$, but knowledge 
of $L$ and $\phi^*$ is not assumed by the algorithm. 
\begin{assumption}\label{ass:smooth_bounded}
	The gradient $\nabla \phi$ is $L$-Lipschitz continuous and $\phi$ is bounded below by some constant $\phi^*$. 
\end{assumption}
We consider a setting where neither the true function value $\phi(x)$ nor the true gradient $\nabla \phi(x)$ can be computed directly. Instead, the algorithm obtains all necessary function-related information by querying zeroth- and first-order stochastic oracles.

A key feature of our framework is its minimal assumptions on the oracles, which are only required to provide a reasonably accurate estimate with a certain probability. This flexibility allows for the modeling of highly unreliable or noisy inputs, a challenge in real-world optimization. 
Our setup is motivated by the prevalence of data anomalies and outliers in stochastic optimization, which can arise from diverse sources such as system errors, malfunctioning or byzantine machines, adversarial attacks and corruption, insufficient sampling, outdated data, or extreme, rare events. When such issues propagate into optimization, they can lead to poor performance and, consequently, suboptimal decision-making. This highlights the critical need for optimization algorithms that can reliably handle inputs containing bias, noise, and even adversarial corruption. Below, we give two motivating examples.

\textbf{Example 1: Expected Risk Minimization in Machine Learning.} 
Consider the stochastic optimization problem of expected risk minimization, where 5\% of the dataset consists of outliers.
In each iteration, a stochastic gradient estimate is obtained by averaging gradients over a randomly sampled mini-batch of data points. With a batch size of 32 (a common choice alongside 64, 128, and larger sizes), there is only about a 19\% probability that the mini-batch contains no outliers. This means that gradient estimates are free from corruption with probability 19\%, while with probability 81\%, the mini-batch contains at least one outlier, potentially corrupting the gradient estimate by an arbitrarily large amount. 
In contrast, function value estimates are often more robust to outliers. When the loss function is bounded (e.g., 0-1 loss), significantly corrupting the function value estimate with a mini-batch of 32 requires multiple outliers. For instance, the probability of sampling five or more outliers is only approximately 2\%. Consequently, the probability of a large error in the function value estimate can be substantially lower than that for gradient estimates.

\textbf{Example 2: Derivative-Free Optimization.} 
% Gradient corruption also frequently arises in derivative-free optimization. 
In many real-world problems -- such as those in chemical engineering and variational quantum algorithms (VQAs) -- gradient information is unavailable, and only noisy or biased estimates of function values can be obtained. 
A common approach in these settings is to estimate gradients using finite-difference methods, which typically require \( n + 1 \) function value samples to construct a gradient estimate in dimension \( n \).
Even under relatively benign function value estimates, this approach can lead to frequent gradient corruption. Suppose each individual function value estimate is accurate with probability 99\%, and is severely corrupted with probability 1\%. In moderate dimensions, such as \( n = 100 \), the probability that all \( n + 1 \) function value estimates are accurate is only approximately $0.99^{101} \approx 36\%$. This leaves a $1 - 36\% = 64\%$ probability that the gradient estimate is significantly corrupted due to at least one inaccurate function value estimate.

\paragraph{The Challenge.} In both settings described above, data anomalies can severely corrupt the inputs to optimization algorithms, highlighting a critical need for methods that are resilient to bias, noise, and adversarial corruption. However, most existing stochastic optimization algorithms, including popular methods like ADAM and AdaGrad, are not designed for such severely corrupted inputs, where gradient estimates may be arbitrarily corrupted with probability greater than 50\% and function value estimates are subject to heavy-tailed noise. Most theoretical analyses typically rely on strong assumptions, such as bounded variance of gradient estimates, which fail to hold in the presence of heavy-tailed noise or adversarial attacks. 

In this work, we introduce an adaptive algorithmic and analytical framework designed for such common yet challenging scenarios. The key idea is to leverage function value information to ensure reliability when gradient estimates are severely corrupted. Although function value estimates themselves can be subject to heavy-tailed noise or adversarial corruption, combining them with gradient estimates improves the overall reliability of the algorithm. 

A key aspect of our framework is that our analysis is conducted with respect to the true, underlying objective function $\phi(x)$ (e.g., the true expected risk in ERM), which we assume is free from data anomalies. All sources of data corruption and noise are modeled as inaccuracies in the oracle outputs. We begin by formalizing the probabilistic oracles that model the noisy and corrupted inputs available to the algorithm.

\subsection{Oracles}
	
{\bf Stochastic Zeroth-Order Oracle (SZO($\epsilon_f, q, \zeta_q$))}
	Given  a point $x$,  the oracle computes  $f(x,\Xi^0(x))$, a (random) estimate of the function value $\phi(x)$. $\Xi^0(x)$ is a random variable (whose distribution may depend on $x$).  
	We assume the absolute value of the estimation error $E(x)=|f(x,\Xi^0(x)) - \phi(x)|$ (we omit the dependence on $\Xi^0(x)$ for brevity) to have bounded expectation and a bounded $q$-th centered moment, for some $q\geq 2$.
		\begin{equation}\label{eq:zero_order_1}
		{\mathbb E_{\Xi^0(x)}}\left [ E(x)\right ]\leq \epsilon_f \ \text{and }\ {\mathbb E_{\Xi^0(x)}} (E(x)-\E[E(x)])^q\leq \zeta_q.
		\end{equation}
	
	The input to the oracle is $x$, the output is $f(x,\Xi^0(x))$ or $f(x,\Xi^0)$ for short, and the values ($\epsilon_f, q, \zeta_q$) are intrinsic to the oracle. 
	
	\begin{remark}
		\begin{itemize}
			\item In the case where $q=2$, the assumption on the zeroth-order oracle reduces to the noise having bounded expectation and bounded variance
			\begin{equation}\label{eq:zero_order_2}
				{\mathbb E}\left [ E(x)\right ]\leq \epsilon_f \ \text{and } \Var(E(x))\leq \zeta_2=\sigma^2.
			\end{equation}
			\item In the case where the estimation error is a subexponential random variable, then the noise has exponentially decaying tail and bounded moments of all orders. In this case, for some parameters $(\nu,b)$, 
				\begin{equation}\label{eq:zero_order_3}
				{\mathbb E}\left [ E(x)\right ]\leq \epsilon_f \ \text{and }\ {\mathbb E}\left [\exp\{\lambda (E(x)-\E[E(x)])\}\right ]\leq \exp\left(\frac{\lambda^2\nu^2}{2}\right), \quad \forall \lambda\in \left[0,\frac{1}{b}\right],
			\end{equation}
			or equivalently, for some parameters $(\lambda,a)$,
		\begin{equation}\label{eq:zero_order_4}
			{\mathbb E}\left [ E(x)\right ]\leq \epsilon_f \ \text{ and } {\mathbb P}\left ( E(x)\geq t\right ) \leq e^{\lambda (a-t)}, 	\text{ for any $t>0$}
		\end{equation}
		or equivalently, for some constant $K\geq 1$,
			\begin{equation}\label{eq:zero_order_5}
		{\mathbb E}\left [ E(x)\right ]\leq \epsilon_f \ \text{and }\ {\mathbb E} (E(x)-\E[E(x)])^q\leq \zeta_q \text{, where } \zeta_q=(Kq)^q \text{, for all } q\geq 1.
			\end{equation}
			This is the zeroth-order oracle assumption used in \cite{ jin2024high, cao2024first, berahas2023sequential,menickelly2023stochastic,scheinberg2023stochastic}.
			For a proof of the equivalence of \eqref{eq:zero_order_3}, \eqref{eq:zero_order_4}, and \eqref{eq:zero_order_5}, see Proposition 2.7.1 in \cite{vershynin2018high}.
		\end{itemize} 
	\end{remark}
	\textbf{Alternative assumption on function value differences.} Instead of assuming bounded moments on individual function value errors, we may consider a more general assumption on the difference of function values. For any two points $x$ and $x^+$, we could assume that the random variable 
	\begin{equation}
		D(x,x^+) = \left|(f(x,\Xi^0) - f(x^+,\Xi^0)) - (\phi(x) - \phi(x^+))\right|
	\end{equation}
	satisfies the bounded moments condition:
	\begin{equation}\label{eq:zero_order_diff}
		{\mathbb E_{\Xi^0}}\left [ D(x,x^+)\right ]\leq \epsilon_f \quad \text{and} \quad {\mathbb E_{\Xi^0}} \left(D(x,x^+)-\E[D(x,x^+)]\right)^q\leq \zeta_q.
	\end{equation}
	This assumption is more general than the individual bounded moments assumption above. It is particularly beneficial when using common random numbers. 
	Although all results extend under this broader assumption, we adopt the individual bounded-moments assumption for clarity of exposition.
	
	In real-life, there may be settings that are adversarial rather than stochastic. To address such settings, we introduce the following corrupted zeroth-order oracle designed for adversarial settings.
	
\textbf{Corrupted Zeroth-Order Oracle (CZO$(\epsilon_f, \epsilon_c,\delta_0)$)}. 
	Given  a point $x$,  the oracle computes  $f(x,\Xi^0(x))$, 
	where $\Xi^0(x)$ is a random variable, whose distribution may depend on $x$, $\epsilon_f, \epsilon_c$ and $\delta_0$, that satisfies 
	\begin{equation}\label{eq:zero_order_cor}
		\P_{\Xi^0(x)}(| \phi(x)-f(x,\Xi^0(x))|\leq  \epsilon_f)\geq 1-\delta_0, \quad
		\P_{\Xi^0(x)}(| \phi(x)-f(x,\Xi^0(x))|\leq  \epsilon_f+\epsilon_c)= 1.
	\end{equation}
	where $\delta_0\in(0,1)$.
	We view  $x$ as the input to the oracle,  $f(x,\Xi^0(x))$ as the output and the values $(\epsilon_f, \epsilon_c,\delta_0)$ as values intrinsic to the oracle. 
	By definition, $|f(x,\Xi^0(x))-\phi(x)|$ is a bounded random variable. 
While the CZO can be viewed as a special case of the SZO---since a bounded random variable has all moments bounded---it is instructive to analyze it separately. The CZO framework is particularly well-suited for modeling noise from adversarial settings or data outliers, as it explicitly distinguishes between two scenarios: 
\begin{itemize}
    \item With probability $1-\delta_0$, the error is small, resulting from typical stochastic sampling or minor bias.
    \item With probability $\delta_0$, the error is large, due to outliers or adversarial corruption.
\end{itemize}
This modeling differs from the SZO, and the distinction is crucial because the oracle parameters have distinct interpretations. For instance, $\epsilon_f$ in the CZO represents a probabilistic error bound, whereas in the SZO it represents a bound on the expected error. 

\begin{remark}[Expected Risk Minimization with Corrupted Data]
An example that motivates the CZO oracle is expected risk minimization (ERM) in machine learning with a corrupted dataset. Consider an objective $\phi(x) = \mathbb{E}_{d \sim \mathcal{D}}[l(x, d)]$, where $l(x,d)$ is the loss for a model with parameters $x$ on a data point $d$. In practice, $\phi(x)$ is estimated by an empirical average over a mini-batch of data, $f(x) = \frac{1}{B}\sum_{i=1}^B l(x, d_i)$.

Now, suppose a fraction of the underlying data distribution is corrupted. When we draw a mini-batch, the number of corrupted samples in it is random. We consider a setting where the loss function is bounded (e.g., 0-1 bounded loss or some other bounded loss), which ensures the error from corrupted samples is also bounded. We can categorize the outcome based on the level of corruption in the mini-batch:
\begin{itemize}
    \item With probability $1-\delta_0$, the number of corrupted data points in the mini-batch is small, such that their effect on the function value estimate is limited. In this case, the estimation error $|f(x) - \phi(x)|$ is bounded by a baseline error term $\epsilon_f$.
    \item With probability $\delta_0$, the mini-batch contains a significant number of corrupted points. This leads to a larger estimation error, which is bounded by $\epsilon_f + \epsilon_c$, where $\epsilon_c$ represents the maximum additional error due to the high level of corruption.
\end{itemize}
This scenario maps directly to the CZO($\epsilon_f, \epsilon_c, \delta_0$) definition, illustrating its utility in modeling practical machine learning problems. Note that under this model, the expected error is bounded by $\epsilon_f+\delta_0\cdot\epsilon_c$.
\end{remark}
		
{\bf Stochastic First-Order Oracle (SFO($r(\cdot),
    \delta_1$))}.
	Given a point $x$ and a parameter $\alpha > 0$ (e.g., a step size or trust-region radius), the oracle computes and outputs $g(x, \Xi^1(x))$ or $g(x, \Xi^1)$ for short, which is a (random) estimate of the 
    gradient $\nabla \phi(x)$. The distribution of the random variable $\Xi^1(x)$ may depend on $x$. We assume that with probability at least $1-\delta_1$, the estimation error is bounded by an accuracy tolerance function $r(\alpha)$:
	\begin{equation}\label{eq:first_order_general}
	    \mathbb{P}_{\Xi^1(x)}\left( \|\nabla \phi(x) - g(x, \Xi^1(x))\| \leq r(\alpha) \right) \geq 1 - \delta_1.
	\end{equation}
	The function $r(\alpha)$ defines the required relative accuracy, which may depend on algorithm parameters like $\alpha$ and intrinsic oracle parameters (related to bias and variance terms). With probability $\delta_1 \in [0,1)$, the gradient estimate may be arbitrarily corrupted.

	\textbf{Connections to Algorithmic Frameworks.} 
        The general form in \eqref{eq:first_order_general} unifies 
        accuracy conditions for many optimization methods in the literature, including:
		\begin{itemize}
			\item \textbf{Trust-Region Methods:} In trust-region-like methods, the model accuracy is typically required to be proportional to the trust-region radius $\alpha$.  For instance, one typical condition \cite{cao2024first} is: 
			\begin{equation} \label{eq:first_order_trust_region}
				\mathbb{P}_{\Xi^1(x)}\left(\|\nabla \phi(x) - g(x, \Xi^1(x))\| \leq \epsilon_g+\kappa\alpha\right)\geq 1-\delta_1.
			\end{equation}
			\item \textbf{Line-Search Methods:} In line-search or step-search methods, the accuracy requirement might take the following form \cite{jin2024high}: 
			\begin{equation}\label{eq:first_order_specific_form}
				{\mathbb P_{\Xi^1(x)}\left(\|g(x,\Xi^1(x))-\nabla \phi
        (x)\|\leq \max \{\epsilon_g, \min\{\tau,
        \kappa \alpha\}\|g(x,
        \Xi^1(x))\|\}\right)\geq 
        1-\delta_1,}
			\end{equation}
			or
			\begin{equation}\label{eq:first_order_specific_form2}
				{\mathbb P_{\Xi^1(x)}\left(\|g(x,\Xi^1(x))-\nabla \phi
        (x)\|\leq \max \{\epsilon_g, \tau\|\nabla \phi(x)\|\}\right)\geq 
        1-\delta_1.}
			\end{equation}
			While the analyses in this paper apply equally well to both \eqref{eq:first_order_specific_form} and \eqref{eq:first_order_specific_form2}, we will primarily work with \eqref{eq:first_order_specific_form} due to its more common use in the literature.
			
		\end{itemize}
		In both examples, the presence of $\epsilon_g \ge 0$ in the accuracy requirement allows for a constant, irreducible amount of bias in the gradient estimation, while $\kappa \ge 0$ and $\tau \ge 0$ are constants controlling the relative accuracy. The values $(\epsilon_g, \tau,\kappa,\delta_1)$ are intrinsic to the oracle. 

\begin{remark}
	By definition, we allow the first-order oracle to be corrupted by an arbitrarily large amount in each iteration with probability $\delta_1$. This model is well-suited for scenarios like Empirical Risk Minimization (ERM) with corrupted data. For instance, consider estimating the gradient from a mini-batch of data where a fraction of the underlying dataset is corrupted by outliers or adversarial examples. The probability of sampling at least one corrupted data point in a mini-batch depends on the corruption ratio and the batch size. If a mini-batch is contaminated, the resulting gradient estimate can be arbitrarily poor, pointing in any direction with any magnitude. This corresponds to the failure case of the SFO, which occurs with probability $\delta_1$. Conversely, if the mini-batch is clean, the gradient estimate is still stochastic, but its error can typically be bounded, satisfying the accuracy condition $\|\nabla \phi(x) - g(x, \Xi^1)\| \le r(\alpha)$ with probability $1-\delta_1$. The SFO framework thus naturally models the behavior of gradient estimation in the presence of data corruption.

	The stochastic derivative-free optimization setting discussed in the introduction provides another natural application of our framework. In the scenario with $n=100$, where each function value estimate is corrupted or has a large noise with probability $\delta_0 = 0.01$, the resulting finite-difference gradient estimate can be arbitrarily corrupted with a probability as high as $\delta_1 = 0.64$. This scenario is precisely what our SZO and SFO definitions are designed to handle.
\end{remark}

To summarize, the key oracle parameters that characterize the quality of the estimates are:
\begin{itemize}
    \item For SZO: $\epsilon_f$ (expected error bound), $q \geq 2$ (moment order), $\zeta_q$ (moment bound)
    \item For CZO: $\epsilon_f$ (baseline error), $\epsilon_c$ (additional corruption), $\delta_0$ (corruption probability)
    \item For SFO: $r(\cdot)$ (accuracy tolerance function), $\delta_1$ (failure probability)
\end{itemize}

\textbf{Optimization in Highly Unreliable or Noisy Environments.} A central challenge addressed in this paper is optimization in what we term a \emph{highly unreliable} or \emph{highly noisy} environment. This setting is characterized by a low probability of obtaining reasonable gradient and function value estimates from the oracles, meaning in every iteration the algorithm must navigate using local models that are more often misleading than helpful. Specifically, we consider scenarios where the joint probability of receiving sufficiently accurate gradient and function value estimates in any given iteration can be less than $\frac 1 2$. Our analysis is specifically designed to ensure reliability in this challenging regime. The formal probabilistic framework and related notions for this setting will be made precise in the subsequent sections.

\textbf{Related literature.} Our work contributes to the growing literature on adaptive optimization methods with stochastic or probabilistic oracles. 
The foundation for these methods was established in \cite{CS17} and \cite{gratton2018complexity}. In \cite{CS17}, the authors analyzed the expected convergence of line-search and cubic regularized Newton methods using probabilistic models for gradient information. 
In \cite{gratton2018complexity}, the authors establish high-probability convergence for a trust-region method \cite{bandeira2014convergence} that permits arbitrarily inaccurate gradient estimates with probability greater than $\frac 1 2$.
However, both of these works assume that all function value estimates are exact.

The line search method was extended by \cite{berahas2021global} to allow for bounded noise in function values, and by \cite{paquette2018stochastic} to a more general setting with unbounded noise with small variance, while under more complex step and accuracy adjustment process and zeroth-order oracle requirements. 
These studies initially concentrated on achieving convergence in expectation. More recently, \cite{jin2024high} introduced the first high-probability iteration complexity bounds for a stochastic step-search method, while \cite{jin2025sample} developed a comprehensive framework for analyzing the expected and high-probability sample complexity of this class of adaptive algorithms. All of these contributions assume that the first- and higher-order oracle outputs are sufficiently accurate with probability greater than $\frac 1 2$ in every iteration.

The stochastic trust-region literature has been developed in parallel. \cite{blanchet2019convergence} established convergence rate in expectation for a version of stochastic trust-region methods. Building on this and \cite{jin2024high}, \cite{cao2024first} established first- and second-order high-probability complexity bounds for trust-region methods with noisy oracles, assuming the oracle outputs are sufficiently accurate with high probability in every iteration. Similar progress has been made for adaptive cubic regularization methods. For instance, \cite{bellavia2020stochastic, bellavia2022adaptive} developed stochastic cubic regularization methods for finite-sum minimization and nonconvex optimization, and \cite{scheinberg2023stochastic} provided a high-probability complexity bound for a stochastic adaptive regularization method with cubics.
Recent work has also extended these adaptive stochastic methods to more specialized settings. \cite{berahas2023sequential} introduced a sequential quadratic programming (SQP) method with high-probability complexity bounds for solving nonlinear equality-constrained stochastic optimization, and \cite{menickelly2023stochastic} developed stochastic quasi-Newton methods with high-probability iteration complexity bounds for settings where utilizing common random numbers in gradient estimation is not feasible.

Importantly, most existing works in the area assume that the first- and higher-order oracle outputs are sufficiently accurate with probability greater than $\frac 1 2$ in every iteration. A notable exception is \cite{gratton2018complexity}, 
which establishes a high-probability iteration complexity bound for a trust-region method that allows the gradient estimate in each iteration to have arbitrarily large error with probability greater than $\frac 1 2$. 
However, a key limitation of \cite{gratton2018complexity} is the assumption that function value estimates are exact. 
In this paper, we tackle the challenging scenario where \emph{both} the gradient estimate can have arbitrarily large error with probability greater than $\frac 1 2$, \emph{and} function value estimates may be corrupted by heavy-tailed noise. 
As a result, we extend the application and scope of previous works by allowing significantly more relaxed assumptions on the inputs of the algorithm.
We develop a unified algorithmic framework that encompasses both trust-region and line-search methods, together with a unified analysis that applies to all algorithms within this framework.

\textbf{Our contribution.} The main contributions of this work are as follows:
\begin{itemize}
	\item \textbf{A Unified Algorithmic Framework:} We introduce a unified algorithmic and analytical framework for adaptive optimization that is resilient to highly unreliable inputs. This framework is general enough to encompass both trust-region and line-search methods and is specifically designed to handle settings where gradient estimates can be arbitrarily corrupted with high probability and function value estimates are subject to heavy-tailed noise.
	\item \textbf{Analysis for Highly Unreliable or Noisy Environments:} We remove the restrictive assumption that in every iteration, a reliable model (i.e. a sufficiently accurate gradient and function value estimate) must be obtained with a probability greater than $\frac{1}{2}$. Our analysis is the first to formally handle scenarios where with high probability (could be close to one), the model can be arbitrarily corrupted in every iteration with the function 
    value estimates subject to heavy-tailed noise, a critical feature for ensuring robustness against data anomalies and corruption.
    \item \textbf{High-Probability Complexity Bounds:} We provide the first iteration complexity analysis for this setting that establishes (overwhelmingly) high-probability bounds on the stopping time. We show that the tail probability of this stopping time decays either exponentially or polynomially, depending on the assumptions on the zeroth-order oracle (e.g., whether the noise is sub-exponential or heavy-tailed).
\end{itemize}

In the remainder of this paper, we will first present the unified algorithmic framework in \Cref{sec:algo}, along with its instantiations for both trust region and line search methods. Following this, \Cref{sec:definitions} will introduce the general stochastic process that underpins the framework, as well as the key definitions and notation used throughout our discussion. Finally, \Cref{sec:main_bound} will detail the main theoretical results, providing high-probability iteration complexity bounds and proofs for the general framework, as well as for its specific instantiations in trust region and line search methods.

\section{Algorithm}\label{sec:algo}

In this section, we present our unified algorithmic framework for adaptive optimization with highly unreliable inputs. The framework handles both gradient and function value noise through the oracles defined in the introduction. It encompasses several algorithms in the literature, including trust region and line search methods, as special cases. 

\subsection{General Algorithmic Framework}

The algorithmic framework employs an adaptive step-size strategy in which the step size parameter is adjusted dynamically in each iteration based on the estimated progress of the algorithm.  
In each iteration, we construct a local model of the objective function using gradient and function value estimates obtained from the oracles. A candidate point is then computed, typically by minimizing this model. The quality of this candidate point is evaluated using a \emph{sufficient descent test}, which compares the estimated reduction in function value (obtained from the zeroth-order oracle) to a quantity related to the reduction in the model value. Since function value estimates are noisy, the sufficient descent test is typically relaxed compared to its deterministic counterpart to account for this noise.

Based on the outcome of the sufficient descent test, the algorithm takes one of three actions:
\begin{itemize}
    \item If the candidate point passes the sufficient descent test and meets additional acceptance criteria, it is accepted and the step size parameter is increased. This encourages larger steps in subsequent iterations, allowing for potentially faster progress. Such steps are called \emph{successful steps}.
    \item If sufficient reduction is achieved but additional acceptance criteria are not met, the candidate point is accepted but the step size parameter is decreased. 
    \item If sufficient reduction is not achieved, the candidate point is rejected and the step size parameter is decreased. 
\end{itemize}
Steps in these latter two categories are called \emph{unsuccessful steps}. Decreasing the step size parameter encourages more conservative steps within a smaller neighborhood where the model is more likely to be accurate. 

Importantly, this mechanism allows the algorithm to dynamically increase or decrease the step size based on the estimated progress of the algorithm from iteration to iteration, allowing it to adapt to the local geometry of the function and the level of noise from the oracles. The framework is presented in \Cref{alg:generic_stochastic}.

\begin{algorithm}[H]
	\caption{General Algorithmic Framework}
	\label{alg:generic_stochastic}
	\begin{algorithmic}[1]
	\State \textbf{Initialization:} Choose parameters $\theta \in (0,1)$ (sufficient reduction parameter), $\gamma_{\mathrm{inc}} > 1$ (step size increase factor), $\gamma_{\mathrm{dec}} \in (0,1)$ (step size decrease factor), initial point $x_0 \in \mathbb{R}^n$, and initial step size $\alpha_0 > 0$. Set $k \leftarrow 0$.
	
	\State \textbf{Main Loop:} For $k = 0, 1, 2, \ldots$
	
	\State \textbf{Step 1: Model Construction and Step Computation}
	
	Query the first-order oracle to obtain gradient estimate $g_k = g(x_k, \Xi_{1,k})$, query the zeroth-order oracle to obtain function value estimate $f_k = f(x_k, \Xi_{0,k})$ at $x_k$, construct a local model $m_k$ of $\phi$ around $x_k$ using $g_k$ and $f_k$, and compute a step $s_k(\alpha_k)$ that achieves sufficient model reduction.
	
	\State \textbf{Step 2: Compute Trial Point and Acceptance Test}

	Set trial point $x_k^+ \leftarrow x_k + s_k(\alpha_k)$, query the zeroth-order oracle to obtain $f_k^+ = f(x_k^+, \Xi_{0,k}^+)$ (function value estimate at $x_k^+$), compute reduction estimate $f_k - f_k^+$, and check sufficient reduction conditions (parameterized by $\theta$) with respect to the predicted model reduction $m_k\left(x_k\right)-m_k\left(x_k^{+}\right)$.
	
	\State \textbf{Step 3: Update Rule}

	\If{sufficient reduction is achieved}
		\State $x_{k+1} \leftarrow x_k^+$ {(accept step)}
		\If{additional acceptance criteria are met}
			\State $\alpha_{k+1} \leftarrow \gamma_{\mathrm{inc}} \alpha_k$ (increase step size)
		\Else
			\State $\alpha_{k+1} \leftarrow \gamma_{\mathrm{dec}} \alpha_k$ (decrease step size)
		\EndIf
	\Else
		\State $x_{k+1} \leftarrow x_k$ (reject step)
		\State $\alpha_{k+1} \leftarrow \gamma_{\mathrm{dec}} \alpha_k$ (decrease step size)
	\EndIf
    
	\State \textbf{Step 4:} Set $k \leftarrow k + 1$ and continue.
	\end{algorithmic}
\end{algorithm}

\vspace{3mm}

The framework above is deliberately general; the specific forms of the model $m_k$, step $s_k(\alpha_k)$, and acceptance criteria depend on the specific instantiations of the framework. We next show how a trust region and line search method fit into this framework.

\subsection{Trust Region Instantiation}

In the trust region approach, the local model $m_k$ is typically as follows:
\begin{equation}
m_k(x_k + s) = f_k + g_k^T s + \frac{1}{2} s^T H_k s,
\end{equation}
where $H_k$ can be any matrix as long as there exists some constant $\kappa_H$ such that $\norm{H_k} \leq \kappa_H$ for all $k$.

The candidate step is computed by approximately minimizing the model in the ball of radius $\alpha_k$ around the current iterate:
\begin{equation}
s_k(\alpha_k) = \arg\min_{s \in B(x_k, \alpha_k)} m_k(x_k + s).
\end{equation}

The accuracy requirement for the first-order oracle in the trust region method, following \cite{cao2024first}, is given by \Cref{eq:first_order_trust_region}, where a smaller trust region radius translates to a higher accuracy requirement.

\begin{algorithm}[H]
\caption{Trust Region Based Algorithm}
\label{alg:trust_region}

\textbf{Initialization:} Starting point $x_0$; initial radius $\alpha_0 > 0$;  $\eta_1 > 0$, $\eta_2 > 0$, $\delta_1 \in [0,1)$, $\kappa_{\mathrm{fcd}} \in(0,1]$, $\gamma_{\mathrm{inc}} > 1$, $\gamma_{\mathrm{dec}} \in (0,1)$, $\epsilon_f > 0$.

\vspace{2mm}

\textbf{For} $k = 0, 1, 2, \ldots$ \textbf{do}

\vspace{1mm}

Compute vector $g_k$ using SFO$(r(\alpha_k), \delta_1)$ and matrix $H_k$ that has a norm bounded by some constant $\kappa_H$.

\vspace{1mm}

Compute $s_k$ by approximately minimizing $m_k$ in $B(x_k, \alpha_k)$ so that it satisfies $$m_k\left(x_k\right)-m_k\left(x_k+s_k\right) \geq \frac{\kappa_{\mathrm{fcd}}}{2}\left\|g_k\right\| \min \left\{\frac{\left\|g_k\right\|}{\left\|H_k\right\|}, \alpha_k\right\}.$$ 

\vspace{1mm}

Compute $f_k$ and $f_k^+$ using Zeroth-order oracle, with input $x_k$, and $x_k + s_k$, and then compute
$$\rho_k = \frac{f_k - f_k^+ + 2\epsilon_f}{m_k(x_k) - m_k(x_k + s_k)}.$$

\vspace{1mm}

\textbf{if} $\rho_k \geq \eta_1$ \textbf{then}

\vspace{1mm}

\hspace{0.5cm} Set $x_{k+1} = x_k + s_k$ and
$$\alpha_{k+1} = \begin{cases}
\gamma_{\mathrm{inc}} \alpha_k & \text{if } \|g_k\| \geq \eta_2 \alpha_k \\
\gamma_{\mathrm{dec}} \alpha_k & \text{if } \|g_k\| < \eta_2 \alpha_k
\end{cases}$$

\vspace{1mm}

\textbf{else}

\vspace{1mm}

\hspace{0.5cm} Set $x_{k+1} = x_k$ and $\alpha_{k+1} = \gamma_{\mathrm{dec}} \alpha_k$.

\vspace{1mm}
% \textbf{end if}
\end{algorithm}

\vspace{2mm}

This algorithm follows the standard trust-region framework, where the additional term $2\epsilon_f$ accounts for noise in function evaluations. It matches Algorithm~1 of \cite{cao2024first}, with the sole difference that the trust-region radius is updated using $\gamma_{\mathrm{inc}}$ and $\gamma_{\mathrm{dec}}$ (rather than $\gamma$ and $\gamma^{-1}$).
In this trust region algorithm, an iteration is defined as successful if the sufficient reduction condition $\rho_k \geq \eta_1$ and the gradient norm condition $\left\|g_k\right\| \geq \eta_2 \alpha_k$ are both satisfied. 

\subsection{Line Search Instantiation}

One can obtain a step search or line search method by setting the model to be
$$m_k(x_k + s) = f_k + g_k^T s + \frac{1}{2\alpha_k} s^T s$$ and $s_k(\alpha_k) = -\alpha_k g_k$.

The accuracy requirement for the first-order oracle in the line search method, following \cite{jin2024high}, is given by \Cref{eq:first_order_specific_form}. A smaller step size parameter translates to a higher accuracy requirement.

\Cref{alg:line_search} is an instantiation of the general framework in \Cref{alg:generic_stochastic} for a line search method. It is identical to the SASS Algorithm 
from \cite{jin2024high} but with an additional criterion of:
$$\norm{g_k} < \epsilon_{\mathrm{rej}},$$
for some $\epsilon_{\mathrm{rej}}>0$. 
The condition $f_k^+ \leq f_k - \alpha_k\theta\|g_k\|^2 + 2\epsilon_f$ checks if sufficient descent is achieved. If it is, the step is accepted. The additional condition $\|g_k\| \geq \epsilon_{\mathrm{rej}}$ then determines whether the step size should be increased. If the sufficent descent condition is not met, or the norm of the gradient estimate is too small, then the iteration is unsuccessful, and the step size is decreased. 
This mechanism prevents the algorithm from increasing the step size when the gradient estimate is small, which might indicate unreliable estimates rather than actual progress towards a solution.

\begin{algorithm}[H]
\caption{Line Search Based Algorithm}
\label{alg:line_search}

\textbf{Initialization:} Choose $\theta \in (0,1)$, $\gamma_{\mathrm{inc}} > 1$, $\gamma_{\mathrm{dec}} \in (0,1)$, $\epsilon_f > 0$, $\epsilon_{\mathrm{rej}} > 0$, $x_0 \in \mathbb{R}^n$, and $\alpha_0  > 0$. Set $k \leftarrow 0$.

\vspace{2mm}

\textbf{For} $k = 0, 1, 2, \ldots$ \textbf{do}

\vspace{1mm}

Compute gradient estimate $g_k$ using SFO$(r(\alpha_k), \delta_1)$. 

\vspace{1mm}

Compute search direction $d_k = -g_k$ and trial point $x_k^+ = x_k + \alpha_k d_k$.

\vspace{1mm}

Compute $f_k$ and $f_k^+$ using Zeroth-order oracle, with input $x_k$ and $x_k^+$.

\vspace{1mm}

\textbf{if} $f_k^+ \leq f_k - \alpha_k\theta\|g_k\|^2 + 2\epsilon_f$ \textbf{then}

\vspace{1mm}

\hspace{0.5cm} Set $x_{k+1} = x_k^+$ and
$$\alpha_{k+1} = \begin{cases}
	\gamma_{\mathrm{inc}}\alpha_k, & \text{if } \|g_k\| \geq {\epsilon_{\mathrm{rej}}} \\
\gamma_{\mathrm{dec}} \alpha_k & \text{if } \|g_k\| < {\epsilon_{\mathrm{rej}}}
\end{cases}$$

\vspace{1mm}

\textbf{else}

\vspace{1mm}

\hspace{0.5cm} Set $x_{k+1} = x_k$ and $\alpha_{k+1} = \gamma_{\mathrm{dec}}\alpha_k$.

\vspace{1mm}
\end{algorithm}

\subsection{Key Adaptive Features}
In the algorithm, both the \textbf{step size parameter} and the \textbf{gradient accuracy requirement} are adaptive. The step size parameter is adjusted based on the estimated progress of the algorithm from iteration to iteration. It can increase or decrease in each iteration, and it is not monotonic. This allows the algorithm to adapt to the local landscape of the function.

In each iteration, there are two reasons the algorithm might reject a candidate step. Either:
\begin{itemize}
    \item The oracle estimates are so inaccurate so that the candidate step $s_k$ is not a descent direction, or
    \item The step size parameter is so large so that taking this step would overshoot.
\end{itemize}
When the algorithm rejects a step, it decreases the step size parameter and computes a new candidate step. This reduction addresses the second potential cause of failure. Additionally, because the accuracy requirement of the first-order oracle is governed by $r(\alpha_k)$, a smaller step size parameter demands higher accuracy from the oracle. Thus, by adaptively adjusting the step size, the algorithm simultaneously addresses both potential causes of failure.

In \Cref{sec:main_bound}, we present a \textbf{unified analysis framework} that allows \emph{any} algorithm within this algorithmic framework to be analyzed using the same theoretical approach, despite differences in the choice of model $m_k$, candidate step calculation, and sufficient descent condition. 

\section{Definitions and Notation}\label{sec:definitions}

We introduce the key definitions and notation used throughout the paper. The algorithm generates a stochastic process, and we define the relevant random variables and concepts needed for the analysis.

\subsection{Stochastic Process and Filtration}

Let $M_k$ denote the collection of random variables $\{\Xi_{0,k}, \Xi_{0,k}^+, \Xi_{1,k}\}$, whose realizations are $\{\xi_{0,k}, \xi_{0,k}^+, \xi_{1,k}\}$, where:
\begin{itemize}
    \item $\Xi_{0,k}$ and $\Xi_{0,k}^+$ are the random variables from the zeroth-order oracle at $X_k$ and $X_k^+$, respectively,
    \item $\Xi_{1,k}$ is the random variable from the first-order oracle at $X_k$.
\end{itemize}

At iteration $k$, we define the following random variables:
\begin{itemize}
    \item $X_k$: the random iterate at step $k$, with realization $x_k$,
    \item $G_k$: the random gradient estimate from the first-order oracle, with realization $g_k$,
    \item $E_k = |f(X_k, \Xi_{0,k}) - \phi(X_k)|$: the random absolute error in the zeroth-order oracle at $X_k$, with realization $e_k$,
    \item $E_k^+ = |f(X_k^+, \Xi_{0,k}^+) - \phi(X_k^+)|$: the random absolute error in the zeroth-order oracle at $X_k^+$, with realization $e_k^+$,
    \item $\mathcal{A}_k$: the random step size parameter at step $k$, with realization $\alpha_k$.
\end{itemize}

The algorithm generates a stochastic process $\{(G_k, E_k, E_k^+, X_k, \mathcal{A}_k)\}$ with realizations $(g_k, e_k, e_k^+, x_k, \alpha_k)$ adapted to the filtration $\{\mathcal{F}_k : k \geq 0\}$, where $\mathcal{F}_k = \sigma(M_0, M_1, \ldots, M_k)$ is the $\sigma$-algebra generated by $M_0, \ldots, M_k$.
All the random variables defined above are measurable with respect to $\mathcal{F}_k$. Note that $X_k$ and $\mathcal{A}_k$ are in addition measurable with respect to $\mathcal{F}_{k-1}$.

\subsection{Key Definitions}

\begin{definition}[True Iteration]
\label{def:true_iteration}
We say that iteration $k$ is \textbf{true} if both the gradient estimate and function value estimates are sufficiently accurate:
\begin{align}
\|\nabla \phi(X_k) - G_k\| &\leq r(\mathcal{A}_k), \quad \text{and} \\
E_k + E_k^+ &\leq 2\epsilon_f,
\end{align}
where $r(\mathcal{A}_k)$ is the accuracy tolerance function from the first-order oracle definition, and $\epsilon_f$ is an upper bound on the expected error of the zeroth-order oracle.
\end{definition}

The probability of an iteration being true, denoted by $p$, is central to our framework. A lower bound on $p$ can be established from the oracle failure probabilities. The first-order oracle fails with probability $\delta_1$. For the corrupted zeroth-order oracle, since $\delta_0$ is the failure probability of a single query, and an iteration requires two CZO queries (at $x_k$ and $x_k^+$), the probability that at least one of these is corrupted can be bounded by $2\delta_0$ using a union bound. 
This provides a lower bound on the probability of a true iteration: $p \ge 1 - \delta_1 - 2\delta_0$.

A key focus of this paper is the challenging regime where this probability $p$ is low, which we formalize as follows.

\begin{definition}[Highly Unreliable Environment]
\label{def:highly_unreliable}
An optimization environment is considered \textbf{highly unreliable} if the probability $p$ of an iteration being true satisfies $p < 1/2$.
\end{definition}
When working with the corrupted zeroth-order oracle, this condition implies $\delta_1 + 2\delta_0 > 1/2$ in terms of the oracle failure probabilities. In such a regime, the local model provided to the algorithm is more likely to be misleading than helpful in any given iteration.

\begin{definition}[Successful Iteration]
\label{def:successful_iteration}
We say that iteration $k$ is \textbf{successful} if the iterate is updated and the step size is increased, i.e., $X_{k+1} = X_k^+$ and $\mathcal{A}_{k+1} \geq \mathcal{A}_k$. Otherwise, the iteration is \textbf{unsuccessful}.
\end{definition}

We define the following indicator random variables:
\begin{itemize}
    \item $I_k := \mathbbm{1}\{\text{iteration $k$ is true}\}$
    \item $\Theta_k := \mathbbm{1}\{\text{iteration $k$ is successful}\}$
\end{itemize}

Next, we define a random variable $U_k$ that measures whether the step size parameter is large or small. The threshold $\bar{\alpha}$ is a parameter that will depend on the algorithm being used; we will later give its concrete definition for the trust region and line search methods.

\begin{definition}[Large and Small Steps]
\label{def:large_small_steps}
For a threshold $\bar{\alpha} > 0$, we define the random variable $U_k$ as:
\begin{align}
U_k = \begin{cases}
0, & \text{if } \max\{\mathcal{A}_k, \mathcal{A}_{k+1}\} \leq \bar{\alpha} \text{ (small step)} \\
1, & \text{otherwise } \text{ (large step)} 
\end{cases}
\end{align}
Note that $U_k$ is measurable with respect to $\mathcal{F}_k$, since $\mathcal{A}_{k+1}$ is completely determined by $\mathcal{A}_k$ and $\Theta_k$.
\end{definition}

Since $\phi$ is non-convex, our goal is to find an approximately stationary point.
\begin{definition}[Stopping Time]
\label{def:stopping_time}
For some $\epsilon > 0$, let $T_\epsilon$ be the first time such that $\|\nabla \phi(X_{T_\epsilon})\| \leq \epsilon$. We refer to $T_\epsilon$ as the \emph{stopping time} of the algorithm.
\end{definition}

Finally, the following measure of progress will be used as a potential function in our analysis.

\begin{definition}[Measure of Progress]
\label{def:progress_measure}
For each $k \geq 0$, define 
$$Z_k := \phi(X_k) - \phi^*,$$
where $\phi^*$ is the infimum of the function $\phi$. Making progress means decreasing $Z_k$.
\end{definition}

\section{Main Bound}\label{sec:main_bound}
We now present our main results. We first introduce some key assumptions and lemmas that will be used in the analysis. In \Cref{subsec:general_bound}, we present the high probability iteration complexity bound for the general algorithmic framework, assuming these assumptions hold. In \Cref{subsec:trust_region}, we prove that the assumptions hold for the trust region method, and specialize the general bound for the trust region method. Then, in \Cref{subsec:line_search}, we prove that the assumptions hold for the line search method, and specialize the general bound for the line search method.

A key building block in our analysis is that iterations are true with some positive probability:

\begin{proposition}[Probabilistic Accuracy]
\label{ass:probabilistic_accuracy}
There exists a constant $p \in (0, 1]$ such that for all $k < T_\epsilon$:
$$\mathbb{P}(I_k = 1 \mid \mathcal{F}_{k-1}) \geq p.$$
\end{proposition}
This condition ensures that, regardless of the algorithm's history, each iteration is true with probability at least $p>0$.
The value of $p$ depends on the specific oracle parameters. It is straightforward to see that:
\begin{itemize}
    \item For SZO + SFO: $p \geq 1 - \delta_1 - \mathbb{P}(E_k+E_k^{+} > 2 \epsilon_f)$
    \item For CZO + SFO: $p \geq 1 - \delta_1 - 2\delta_0$
\end{itemize}

The following lemma will be useful in the high probability iteration result with the stochastic zeroth-order oracle.

\begin{lemma}[Fuk--Nagaev inequality]\label{lem:moment}
Let $Y_1,\dots,Y_t$ be independent random variables with $\E Y_i=0$ and $\E|Y_i|^q<\infty$ for some $q\ge 2$. Define $S_t=\sum_{i=1}^t Y_i$. Then, for any $u>0$,
$$
\P\big(S_t \ge u\big)\;\le\;\exp\!\left(-\,\frac{u^2}{2\,V^2}\right)\;+\;\frac{C_q}{u^q},
$$
where
$$
V^2:=\frac{1}{4}(q+2)^2 e^q \sum_{i=1}^t \E Y_i^2,
\qquad
C_q:=\left(1+\frac{2}{q}\right)^q \sum_{i=1}^t \E|Y_i|^q.
$$
In particular, if $Y_i$ are i.i.d. with $\E Y_1^2\le \zeta_2$ and $\E|Y_1|^q\le \zeta_q$, then
$$
\P\big(S_t \ge u\big)\;\le\;\exp\!\left(-\,\frac{2\,u^2}{(q+2)^2 e^q \zeta_2\,t}\right)\;+\;\frac{\left(1+\tfrac{2}{q}\right)^q\,\zeta_q\,t}{u^q}.
$$
\end{lemma}

See, e.g., \cite[Corollary~2.5]{JMAA2016} for this one-sided form with explicit constants, and also \cite{FukNagaev1971,Petrov1995} for classical statements and proofs.

In our general algorithmic framework, we make two key assumptions. We will later show that these assumptions hold for both the trust-region and line search algorithms as instantiated in \Cref{sec:algo}.
\begin{assumption}
	\label{ass:small_true_successful}
	Any iteration $k < T_\epsilon$ that is small and true must be successful.
\end{assumption}

\begin{assumption}
	\label{ass:large_step_makes_progress}
	There exists a function $h: \mathbb{R} \to \mathbb{R}$ such that for any large and successful iteration, the function value decreases by at least $h(\epsilon) -(2\epsilon_f+E_k+E_k^+)$.
	\end{assumption}

\begin{lemma}
\label{lem:main_bound1}
Under \Cref{ass:smooth_bounded} and \Cref{ass:small_true_successful}, for any $t\leq T_\epsilon$, \Cref{alg:generic_stochastic} satisfies
$$\sum_{k=0}^{t-1} I_k 
\leq \frac{\lfloor m \rfloor}{\lfloor m \rfloor+1} \sum_{k=0}^{t-1} U_k \Theta_k+\frac{\lfloor m \rfloor}{\lfloor m \rfloor+1} \sum_{k=0}^{t-1} U_k\left(1-\Theta_k\right)+\frac{1}{\lfloor m \rfloor+1} t,$$ where $m=-\frac{\ln\gamma_{\mathrm{inc}}}{\ln\gamma_{\mathrm{dec}}}$.
\end{lemma}
\begin{proof}
	Observe that 
	\begin{align*}
		\sum_{k=0}^{t - 1} I_k 
		&= \sum_{k=0}^{t - 1} I_k U_k +  \sum_{k=0}^{t - 1} I_k(1 - U_k) \\
		&\leq \sum_{k=0}^{t - 1} U_k + \sum_{k=0}^{t - 1} I_k(1 - U_k). \\
	\end{align*}
	We first bound the second term $\sum_{k=0}^{t - 1}I_k(1 -U_k)$, which is the total number of small and true iterations. By \Cref{ass:small_true_successful}, if an iteration is small and true, it must be a successful iteration. This implies $$\sum_{k=0}^{t - 1}I_k(1 - U_k) \leq \sum_{k=0}^{t - 1}\Theta_k(1 - U_k).$$
	In other words, the number of steps that are small and true is upper bounded by the number of steps that are small and successful.
	As a result,
	\begin{align}\label{eq:main_bound1_1}
		\sum_{k=0}^{t - 1} I_k 
		\leq \sum_{k=0}^{t - 1} U_k + \sum_{k=0}^{t - 1} I_k(1 - U_k) 
		 \leq  \sum_{k=0}^{t - 1} U_k + \sum_{k=0}^{t - 1} \Theta_k(1 - U_k).
	\end{align}
   By the dynamics of how the the step size changes, we have the following inequality: 
	$$	
\left\lfloor m \right\rfloor \sum_{k=0}^{t - 1}\Theta_k(1 -U_k)
	\leq \sum_{k=0}^{t - 1}(1-\Theta_k)(1 -U_k).
	$$
Using the inequality above, we have:
	$$t- \sum_{k=0}^{t - 1}U_k
	=\sum_{k=0}^{t - 1}(1-U_k)
	=\sum_{k=0}^{t - 1}\Theta_k(1-U_k)+
	\sum_{k=0}^{t - 1}(1-\Theta_k)(1-U_k)
	\geq (\lfloor m \rfloor+1)\sum_{k=0}^{t - 1}\Theta_k(1-U_k),
	$$
	which implies
	$$
	\sum_{k=0}^{t - 1}\Theta_k(1-U_k)
	\leq 
	\frac{1}{\lfloor m \rfloor+1} \left(t- \sum_{k=0}^{t - 1}U_k\right).
	$$
	Putting \Cref{eq:main_bound1_1} and the inequality above together, we have:
	\begin{align*}
		\sum_{k=0}^{t - 1} I_k 
		 \leq \sum_{k=0}^{t - 1} U_k + \sum_{k=0}^{t - 1} \Theta_k(1 - U_k)
		 \leq \sum_{k=0}^{t - 1} U_k +
		\frac{1}{\lfloor m \rfloor+1} \left(t- \sum_{k=0}^{t - 1}U_k\right)
		=  \frac{\lfloor m \rfloor}{\lfloor m \rfloor+1} \sum_{k=0}^{t - 1} U_k +
		\frac{1}{\lfloor m \rfloor+1} t
		.
	\end{align*}
	
	We conclude that
	$$	\sum_{k=0}^{t - 1} I_k 
	\leq  \frac{\lfloor m \rfloor}{\lfloor m \rfloor+1} \sum_{k=0}^{t - 1} U_k +
	\frac{1}{\lfloor m \rfloor+1} t
	=
	\frac{\lfloor m \rfloor}{\lfloor m \rfloor+1} \sum_{k=0}^{t - 1}U_k\Theta_k 
	+ \frac{\lfloor m \rfloor}{\lfloor m \rfloor+1}\sum_{k=0}^{t - 1}U_k(1 - \Theta_k) +
	\frac{1}{\lfloor m \rfloor+1} t.$$	
\end{proof}

Without loss of generality, we assume that the algorithm starts with a step size parameter $\alpha_0 \geq \bar{\alpha}$. This is because if the initial step size parameter is less than $\bar{\alpha}$, we can simply redefine $\bar{\alpha}$ to be the initial step size parameter, which is still a constant.

\begin{lemma}
\label{lem:main_bound2}
Under \Cref{ass:smooth_bounded}, for any $t\leq T_\epsilon$, \Cref{alg:generic_stochastic} satisfies
$$\sum_{k=0}^{t-1} U_k\left(1-\Theta_k\right)\leq \lceil m\rceil \sum_{k=0}^{t-1} U_k \Theta_k+\left\lceil{\frac{\ln \bar{\alpha} - \ln \alpha_0}{\ln \gamma_{\mathrm{dec}}}}\right\rceil.$$
\end{lemma}
\begin{proof}
	Recall that after any unsuccessful iteration, we decrease the step size by a factor of $\gamma_{\mathrm{dec}}$. On the other hand, after any successful iteration, we increase the step size by a factor of $\gamma_{\mathrm{inc}}$. As a result, {if there are many large unsuccessful iterations, there must also be many large successful iterations}. 
	In particular, since we assume that the initial step size $\alpha_0\geq \bar\alpha$, by the dynamics of the step size, we have that
	$$\sum_{k=0}^{t-1} U_k\left(1-\Theta_k\right) \leq \lceil m \rceil \cdot \sum_{k=0}^{t-1} U_k \Theta_k + \left\lceil{\frac{\ln \bar{\alpha} - \ln \alpha_0}{\ln \gamma_{\mathrm{dec}}}}\right\rceil.$$
	In other words, the total number of large and unsuccessful iterations is within a multiplicative factor away from the total number of large and successful iterations, up to an additive constant. 		
\end{proof}
We now prove an inequality that will be essential for our iteration complexity bound.

\begin{proposition}
\label{thm:main_bound}
Under \Cref{ass:smooth_bounded,ass:small_true_successful,ass:large_step_makes_progress}, for any $t\leq T_\epsilon$, \Cref{alg:generic_stochastic} satisfies
$$\sum_{k = 0}^{t - 1} I_k 
\leq  \frac{\lfloor m \rfloor(\lceil m \rceil+1)}{\lfloor m \rfloor +1}\cdot  \frac{Z_0+\sum^{t-1}_{k=0} (2\epsilon_f+E_k+E_k^+)}{h(\epsilon)}+ \frac{\lfloor m \rfloor}{\lfloor m \rfloor+1} 
\left\lceil{\frac{\ln \bar{\alpha} - \ln \alpha_0}{\ln \gamma_{\mathrm{dec}}}}\right\rceil 
+ \frac{1}{\lfloor m \rfloor+1} t.$$
In particular, if $m = -\frac{\ln\gamma_{\mathrm{inc}}}{\ln\gamma_{\mathrm{dec}}}$ is an integer, we have:
$$\sum_{k = 0}^{t - 1} I_k 
\leq  m\cdot  \frac{Z_0+\sum^{t-1}_{k=0} (2\epsilon_f+E_k+E_k^+)}{h(\epsilon)}+ \frac{m}{ m+1} 
\left\lceil{\frac{\ln \bar{\alpha} - \ln \alpha_0}{\ln \gamma_{\mathrm{dec}}}}\right\rceil 
+ \frac{1}{m+1} t.$$
\end{proposition}
\begin{proof}

We first bound the total number of steps that are large and successful. Note that by the design of the algorithm, a step can increase the function value by at most $2\epsilon_f+E_k+E_k^+$ in each iteration due to the noise in the zeroth order oracle, and by \Cref{ass:large_step_makes_progress} any large and successful step decreases the function value by at least $h(\epsilon)-(2\epsilon_f+E_k+E_k^+)$.
As a result, there can be at most $ \frac{Z_0+\sum^{t-1}_{k=0} (2\epsilon_f+E_k+E_k^+)}{h(\epsilon)}
 $ iterations that are both large and successful.

Together with \Cref{lem:main_bound1} and \Cref{lem:main_bound2}, we conclude that
$$\sum_{k = 0}^{t - 1} I_k 
\leq  \frac{\lfloor m \rfloor(\lceil m \rceil+1)}{\lfloor m \rfloor +1}\cdot  \frac{Z_0+\sum^{t-1}_{k=0} (2\epsilon_f+E_k+E_k^+)}{h(\epsilon)}+ \frac{\lfloor m \rfloor}{\lfloor m \rfloor+1} 
\left\lceil{\frac{\ln \bar{\alpha} - \ln \alpha_0}{\ln \gamma_{\mathrm{dec}}}}\right\rceil 
+ \frac{1}{\lfloor m \rfloor+1} t.$$
 
\end{proof}

\subsection{Iteration Complexity for General Framework}
\label{subsec:general_bound}

For the rest of the paper, we will assume for simplicity that $m$ is an integer. This is to avoid having to carry around $\lfloor m \rfloor$ and $\lceil m \rceil$. 
We can now obtain the high probability iteration complexity bound as follows for the general algorithmic framework.

\begin{theorem}[General high probability iteration complexity]
\label{thm:hp_iter_bound}
Suppose \Cref{ass:smooth_bounded,ass:small_true_successful,ass:large_step_makes_progress} hold. Suppose $p>p_m$, where $p_m = \frac{1}{m+1}+\frac{2 m \epsilon_f+m \mu}{h(\epsilon)}$, and $\mu := \sup_{k} \mathbb{E}[E_k+E_k^{+}]$ denote a uniform upper bound on the expected error across all iterations.
Then 
for any $s\ge 0$, any $\hat p\in\big(p_m + \frac{s}{h(\epsilon)},\; p\big)$, and any
$$
t>\frac{R}{\hat p - p_m - \frac{s}{h(\epsilon)}},
% \\\frac{R}{\hat p - \tfrac12 - \tfrac{r(\epsilon_f,2\epsilon_f)+s}{h(\epsilon)}}
$$
the iteration complexity of \Cref{alg:generic_stochastic} satisfies
$$
\P\big(T_\epsilon \le t\big)\;\ge\; 1 
- \exp \left(-\frac{(p-\hat{p})^2}{2 p^2} t\right) - \P(\bar{B}).
$$ 
Here, the event 
$B = \{ \sum_{k=0}^{t-1} (2\epsilon_f + E_k + E_k^+) \leq t(2\epsilon_f+\mu+s)\}$, $\bar B$ is the complement of $B$, $R =\frac{mZ_0}{h(\epsilon)}+ d$, and $d = \frac{m}{m+1}\lceil\frac{\ln \bar\alpha - \ln \alpha_0}{\ln \gamma_{\mathrm{dec}}}\rceil$.
\end{theorem}

\begin{proof}
Define the event $A = \{\sum_{k=0}^{t-1} I_k \geq\hat{p}t\}$, and $B = \{ \sum_{k=0}^{t-1} (2\epsilon_f + E_k + E_k^+) \leq t(2\epsilon_f+\mu+s)\}$.
By Azuma--Hoeffding inequality %\cite{azuma1967weighted}
applied to the submartingale $\sum_{k=0}^{t-1} I_k-pt$, we have for all $t\geq 1$, and any $\hat{p} \in [0, p)$, 
\begin{align}\label{eq:azuma_hoeffding}
	\P(\bar{A})=\P\left(\sum_{k=0}^{t-1} I_k< \hat{p}t\right) \leq \exp\left( - \frac{(p-\hat{p})^2}{2p^2} t \right).
\end{align}

Consider some positive integer $t > \frac{R}{\hat p - p_m - \frac{s}{h(\epsilon)}}$.
We will show that $\P(T_\epsilon > t, A, B)=0$ by contradiction.
Suppose $T_\epsilon > t$. Then by \Cref{thm:main_bound}, we have:
$$\sum_{k = 0}^{t - 1} I_k 
\leq  m\cdot  \frac{Z_0+\sum^{t-1}_{k=0} (2\epsilon_f+E_k+E_k^+)}{h(\epsilon)}+ d
+ \frac{1}{m+1} t.$$
Furthermore, under the events of $A$ and $B$, we have: $\sum_{k=0}^{t-1} I_k\geq \hat{p}t$ and  $\sum_{k=0}^{t-1} (2\epsilon_f + E_k + E_k^+) \leq t(2\epsilon_f+\mu+s)$. 
Putting them together, we get
$t\leq \frac{R}{\hat p - p_m - \frac{s}{h(\epsilon)}}$.
This contradicts $t>\frac{R}{\hat p - p_m - \frac{s}{h(\epsilon)}}$. As a result, $\P(T_\epsilon > t, A, B)=0$.

Using this fact, we obtain:
\begin{align*}
	\P(T_\epsilon > t) &= \P(T_\epsilon > t, A) + \P(T_\epsilon > t, \bar{A}) \\
	&\leq  \P(T_\epsilon > t, A) + \P(\bar{A}) \\
	&=  \P(T_\epsilon > t, A, B) + \P(T_\epsilon > t, A, \bar{B}) + \P(\bar{A}) \\
	&\leq  \P(T_\epsilon > t, A, B) + \P( \bar{B}) + \P(\bar{A}) \\
	&=  0 + \P( \bar{B}) + \P(\bar{A}).
\end{align*}

As a result, for any $t>
\frac{R}{\hat p - p_m - \frac{s}{h(\epsilon)}},$ with probability at least $1- \P( \bar{A}) - \P(\bar{B})$
 we have  $T_{\epsilon}\leq t$, i.e., the stopping time has been reached. Together with \Cref{eq:azuma_hoeffding}, the result follows.
\end{proof}

\paragraph{Discussion of Theorem~\ref{thm:hp_iter_bound}} The general bound above applies to any algorithm that falls under the framework of \Cref{alg:generic_stochastic} as long as \Cref{ass:smooth_bounded,ass:small_true_successful,ass:large_step_makes_progress} hold.

For the requirement $p > p_m$, let's first consider what this means in idealized setting with exact functions values, where $\epsilon_f=0$ and $\mu=0$. In this case, the condition simplifies to $p > \frac{1}{m+1}$. Recalling that $m = -\frac{\ln\gamma_{\mathrm{inc}}}{\ln\gamma_{\mathrm{dec}}}$, this is equivalent to requiring $p\ln \gamma_{\mathrm{inc}}+(1-p)\ln\gamma_{\mathrm{dec}}>0$. 
This inequality provides a guideline for setting the step-size update parameters: the expected change in the step size parameter, $\alpha_k$, must ensure an ``upward drift" and prevent it from shrinking to zero. A key challenge is that $p$ is usually unknown. In practice, one could simply set $\gamma_{\mathrm{inc}}$ and $\gamma_{\mathrm{dec}}$ so that the inequality is satisfied even for a conservative lower bound of $p$.  That said, the tighter lower bound we have for $p$, the more efficent the algorithm will behave.

When function values are inexact, the requirement $p > p_m$ is related to a lower bound on the achievable target accuracy $\epsilon$;
there exists a problem-dependent lower bound determined by the intrinsic oracle bias and noise levels (captured by $\epsilon_f$ and $\epsilon_g$). For trust region and line search algorithms, $\epsilon$ scales on the order of $\max \left\{\epsilon_g, \sqrt{\epsilon_f}\right\}$. Precise formulas are provided in the trust-region and line-search sections. We will see later that the condition $ p > p_m$  implies the size of the convergence neighborhood of the algorithm is determined by the amount of noise in the oracles.

\subsubsection{General Iteration Complexity with Stochastic Zeroth-Order Oracle}

We now specialize the general bound to the case where the zeroth-order 
oracle is SZO. We first define the event $B_{\rm SZO}(s,t)$ that is related to the accumulated function-value noise. For any $s>0$ and horizon $t\ge 1$, let
$$
B_{\rm SZO}(s,t) := \Big\{ \sum_{k=0}^{t-1} (2\epsilon_f+E_k+E_k^+)\le t(4\epsilon_f+s) \Big\}. 	
$$
Since in the SZO setting $\E[E_k+E_k^+]\le 2\epsilon_f$, we may use a concentration inequality to yield a tail bound $\delta_t(s)$ such that $\P(\bar B_{\rm SZO}(s,t))\le \delta_t(s)$. The specific form of the tail bound is dependent on the distribution of the zeroth-order noise.

In general, if the zeroth-order noise is heavy-tailed with bounded $q$-th centered moment as in the SZO definition, then we can utilize \Cref{lem:moment} to obtain a tail bound. Define the centered variables
$$
Y_k \;:=\; (E_k+E_k^+)-\E[E_k+E_k^+] \qquad (k=0,\dots,t-1),
$$
which satisfy $\E[Y_k]=0$. Together with the SZO moment bound, we have
$$
\E|Y_k|^q \;\le\; 2^{q-1}\Big(\E|E_k-\E E_k|^q+\E|E_k^+-\E E_k^+|^q\Big)\;\le\;2^q\,\zeta_q,\quad
\E Y_k^2 \;\le\; 4\,\zeta_2.
$$
Let $S_{t}:=\sum_{k=0}^{t-1}Y_k$. Since $\E[E_k+E_k^+]\le 2\epsilon_f$, the event $\bar B_{\rm SZO}(s,t)$ implies $S_t > t s$. Applying \Cref{lem:moment} yields
$$
\P\!\left(\bar B_{\rm SZO}(s,t)\right)\;\le\;\P\Big(S_t \ge ts\Big)
\leq \exp\!\left(-\,\frac{s^2}{2(q+2)^2 e^q \zeta_2}\,t\right)\;+\;\frac{\left(1+\tfrac{2}{q}\right)^q 2^q \zeta_q}{s^q\,t^{\,q-1}}.
$$

Applying \Cref{thm:hp_iter_bound} with event $B=B_{\rm SZO}(s,t)$, we obtain the following theorem for the case where the zeroth-order oracle is SZO.
\begin{theorem}[High probability iteration complexity with SZO]
	\label{thm:hp_iter_bound_szo}
	Suppose \Cref{ass:smooth_bounded,ass:small_true_successful,ass:large_step_makes_progress} hold, and $p>p_m^{\rm SZO}$, where $p_m^{\rm SZO} = \frac{1}{m+1}+\frac{4 m \epsilon_f}{h(\epsilon)}$. 
	Then 
	for any $s\ge 0$, any $\hat p\in\big(p_m^{\rm SZO} + \frac{s}{h(\epsilon)},\; p\big)$, and any
	$$
	t>\frac{R}{\hat p - p_m^{\rm SZO} - \frac{s}{h(\epsilon)}},
	$$
	the iteration complexity of \Cref{alg:generic_stochastic} when using SZO satisfies
	$$
	\P\big(T_\epsilon \le t\big)\;\ge\; 1 
	- \exp \left(-\frac{(p-\hat{p})^2}{2 p^2} t\right) \;-\; \left[\,\exp\!\left(-\,\frac{s^2}{2 (q+2)^2 e^q \zeta_2}\,t\right)\;+\;\frac{\left(1+\tfrac{2}{q}\right)^q 2^q \zeta_q}{s^q\,t^{\,q-1}}\right].
	$$ 
	Here, $R =\frac{mZ_0}{h(\epsilon)}+ d$, and $d = \frac{m}{m+1}\lceil\frac{\ln \bar\alpha - \ln \alpha_0}{\ln \gamma_{\mathrm{dec}}}\rceil$.
	\end{theorem}

\begin{remark}
In the following special cases of the stochastic zeroth-order oracle, we can obtain tighter probability bounds:
\begin{itemize}
    \item $q=2$ (bounded variance; see definition as in \Cref{eq:zero_order_2}). Using Chebyshev inequality in place of \Cref{lem:moment}, one can show that
    $$
    \P(T_\epsilon \le t) \ge 1 - \exp\!\left( - \frac{(p-\hat{p})^2}{2p^2} \, t \right) - \frac{2\sigma^2}{s^2 t}.
    $$
    \item $q=\infty$ (sub-exponential tails; see definition as in \Cref{eq:zero_order_3}). Using Bernstein inequality in place of \Cref{lem:moment}, one can show that
    $$
    \P(T_\epsilon \le t) \ge 1 - \exp\left( - \frac{(p-\hat{p})^2}{2p^2} \, t \right) - \exp\left(-\min\left\{\tfrac{s^2}{8\nu^2},\tfrac{s}{4b}\right\}\,t\right).
    $$
	In other words, if the noise in the zeroth-order oracle follows a light-tailed distribution, then the failure probability decays exponentially in $t$.
\end{itemize}
\end{remark}

\subsubsection{General Iteration Complexity with Corrupted Zeroth-Order Oracle}

We now specialize the general bound to the corrupted zeroth-order oracle (CZO) defined in \Cref{eq:zero_order_cor}. Under CZO, we have $\E[E_k+E_k^+]\le 2\epsilon_f+2\delta_0\epsilon_c$. Define, for any $s>0$ and horizon $t\ge 1$, the event
$$
B_{\rm CZO}(s,t) := \Big\{ \sum_{k=0}^{t-1} (2\epsilon_f+E_k+E_k^+)\le t\big(4\epsilon_f+2\delta_0\epsilon_c+s\big) \Big\}.
$$
Since any function value corruption lies in $[0,\,\epsilon_f+\epsilon_c]$ almost surely, Hoeffding's inequality yields the tail bound
$$
\P\big(\bar B_{\rm CZO}(s,t)\big)
\;=\; \P\!\left(\sum_{k=0}^{t-1} (2\epsilon_f+E_k+E_k^+)> t\big(4\epsilon_f+2\delta_0\epsilon_c+s\big)\right)
\;\le\; \exp\!\left(-\frac{s^2\, t}{2(\epsilon_f+\epsilon_c)^2}\right).
$$

Applying \Cref{thm:hp_iter_bound} with $B=B_{\rm CZO}(s,t)$ gives the following result.
\begin{theorem}[High probability iteration complexity with CZO]
\label{thm:hp_iter_bound_czo}
Suppose \Cref{ass:smooth_bounded,ass:small_true_successful,ass:large_step_makes_progress} hold, and
$
p\;>\; p_m^{\rm CZO}$, where $p_m^{\rm CZO} 
:= \frac{1}{m+1}+\frac{m\big(4\epsilon_f+2\delta_0\epsilon_c\big)}{h(\epsilon)}.
$
Then 
for any $s\ge 0$, any $\hat p\in\big(p_m^{\rm CZO} + \tfrac{s}{h(\epsilon)},\; p\big)$, and any
$$
t\;>\; \frac{R}{\hat p - p_m^{\rm CZO} - \tfrac{s}{h(\epsilon)}}
$$
the iteration complexity of \Cref{alg:generic_stochastic} when using CZO satisfies
$$
\P\big(T_\epsilon \le t\big)\;\ge\; 1 
- \exp\!\left(-\frac{(p-\hat{p})^2}{2 p^2}\, t\right) 
- \exp\!\left(-\frac{s^2}{2(\epsilon_f+\epsilon_c)^2}\, t\right).
$$
Here, $R =\tfrac{mZ_0}{h(\epsilon)}+ d$ and $d = \tfrac{m}{m+1}\left\lceil\tfrac{\ln \bar\alpha - \ln \alpha_0}{\ln \gamma_{\mathrm{dec}}}\right\rceil$.
\end{theorem}

\subsection{Iteration Complexity in Expectation and Almost Sure Convergence}
The high-probability bounds in \Cref{thm:hp_iter_bound_szo,thm:hp_iter_bound_czo} imply the following bounds in expectation.

\begin{corollary}[Iteration complexity in expectation]
	\label{cor:iter_complexity_expectation}
	Suppose \Cref{ass:smooth_bounded,ass:small_true_successful,ass:large_step_makes_progress} hold, and let $p>p_m$ as in \Cref{thm:hp_iter_bound}. 
	Fix any $s\ge 0$, $\epsilon > 0$, $\hat p\in\big(p_m + \tfrac{s}{h(\epsilon)},\; p\big)$, and define
	$$
	t_s \,=\, \left\lceil \frac{R}{\hat p - p_m - \tfrac{s}{h(\epsilon)}} \right\rceil,
	\qquad
	c \,=\, \frac{(p-\hat p)^2}{2 p^2}, 
	$$
	where $R$ is as in \Cref{thm:hp_iter_bound}.
	If the zeroth-order oracle is SZO with bounded $q$-th centered moment with $q>2$, or is CZO, then
	$$
	\E[T_\epsilon] \,=\, O(t_s).
	$$
\end{corollary}
\begin{proof}
	By $\E[T_\epsilon]=\sum_{t=0}^\infty \P(T_\epsilon>t)$ and \Cref{thm:hp_iter_bound}, we have
	$$
	\E[T_\epsilon]=\sum_{t=0}^\infty \P(T_\epsilon>t)\;\le\; \sum_{t=0}^\infty [e^{-c t} + \P\!\big(\bar B(s,t)\big)].
	$$
	The tail is summable (polynomial with exponent $q-1>1$ for SZO; exponential for CZO). Hence, $\E[T_\epsilon]\le t_s + O(1)=O(t_s)$.
\end{proof}

For the general framework, using either SZO or CZO also ensures almost sure convergence.

\begin{theorem}[Almost sure convergence]
	Suppose the zeroth-order oracle is SZO and the conditions of \Cref{thm:hp_iter_bound_szo} hold, or the zeroth-order oracle is CZO and the conditions of \Cref{thm:hp_iter_bound_czo} hold. The algorithm reaches an $\epsilon$-stationary iterate in finite time almost surely. That is,
	
$$
\P\left[\bigcap_{k=1}^{\infty} \{T_{\epsilon}>k\}\right] = 0.
$$
\end{theorem}
\begin{proof}
	We see that 
	$\mathbb{P}\left[\cap_{k=1}^{\infty} \left(T_{\epsilon}>k\right)\right]=0$ since the failure probability $\mathbb{P} \left(T_{\epsilon}>k\right)$ is 
	going to $0$ as $k\rightarrow \infty$, for both cases.

\end{proof}

\subsection{Trust Region}
\label{subsec:trust_region}

In this section, we specialize our general analytical framework for the trust region method (\Cref{alg:trust_region}). To apply our main complexity theorems, we will define the concrete forms of the key analytical quantities from our general framework: the small-step threshold $\bar{\alpha}$, the progress measure $h(\epsilon)$, and the lower bound on $\epsilon$.

First, we define the threshold $\bar{\alpha}$ that distinguishes between ``small" and ``large" trust regions. This threshold is carefully chosen to be small enough to satisfy \Cref{ass:small_true_successful}—ensuring that any ``small and true" iteration is guaranteed to be a ``successful" iteration. Its specific form is:
$$
\bar\alpha := \min \left\{\frac{(1-\eta_1)\kappa_{\mathrm{fcd}}(1-\eta)-2\eta}{L+\kappa_{\mathrm{H}}+2 \kappa_{\mathrm{}}+\left(1-\eta_1\right) \kappa_{\mathrm{fcd}} \kappa}, \frac{1-\eta}{\kappa_{\mathrm{}}+\eta_2} \right\} \epsilon,
$$
for some $\eta \in \left(0, \frac{\left(1-\eta_1\right) \kappa_{\mathrm{fcd}}}{\left(1-\eta_1\right) \kappa_{\mathrm{fcd}}+2}\right)$.

By the definition of a large step, if an iteration $k$ is large and successful, then its step size $\mathcal{A}_k$ must satisfy $\mathcal{A}_k \geq \frac{1}{\gamma_{\text {inc }}} \bar{\alpha}$. For the trust region method, we will show a large and successful iteration provides progress proportional to the square of the trust-region radius $\alpha_k^2$, up to the noise related term $2\epsilon_f+E_k+E_k^{+}$. This guarantees any large and successful iteration provides progress of at least $h(\epsilon)-\left(2 \epsilon_f+E_k+E_k^{+}\right)$, where $h(\epsilon)$ is defined as: 
\begin{equation}\label{eq:h_epsilon_tr}
h(\epsilon) = C_{\text{prog}} \left(\frac{\bar\alpha}{\gamma_{\text{inc}}}\right)^2 = C_{\text{prog}} \left( \frac{1}{\gamma_{\text{inc}}} \min \left\{\frac{(1-\eta_1)\kappa_{\mathrm{fcd}}(1-\eta)-2\eta}{L+\kappa_{\mathrm{H}}+2 \kappa_{\mathrm{}}+\left(1-\eta_1\right) \kappa_{\mathrm{fcd}} \kappa}, \frac{1-	\eta}{\kappa_{\mathrm{}}+\eta_2} \right\} \right)^2 \epsilon^2,
\end{equation}
where $C_{\text{prog}} = \frac{1}{2} \eta_1 \eta_2 \kappa_{\mathrm{fcd}} \min \left\{\frac{\eta_2}{\kappa_{\mathrm{H}}}, 1\right\}$. Note that $h(\epsilon)$ is of the order $O(\epsilon^2)$.

Finally, due to the presence of bias in the oracle outputs, the achievable target accuracy $\epsilon$ is inherently limited. Specifically, $\epsilon$ is lower bounded by a certain threshold determined by the magnitude of the bias in zeroth- and first-order oracles. The specific lower bound on $\epsilon$ is given in \Cref{ineq:eps_lowerbound_tr}.

\begin{inequality}[Lower bound on $\epsilon$ for Trust Region]
\label{ineq:eps_lowerbound_tr} \phantom{a}
\begin{itemize}
    \item \textbf{For SZO Oracle:}
    $$ \epsilon > \max\left\{\frac{\epsilon_g}{\eta}, \sqrt{\frac{4m\,\epsilon_f}{C_{\text{prog}}\,\big(p-\tfrac{1}{m+1}\big)}}\right\}. $$
    \item \textbf{For CZO Oracle:}
    $$ \epsilon > \max\left\{\frac{\epsilon_g}{\eta}, \sqrt{\frac{m(4\epsilon_f + 2\delta_0\epsilon_c)}{C_{\text{prog}}\,\big(p-\tfrac{1}{m+1}\big)}}\right\}. $$
\end{itemize}
	for some $\eta \in (0, \frac{(1-\eta_1)\kappa_{\mathrm{fcd}}}{2 + (1-\eta_1)\kappa_{\mathrm{fcd}}})$ and $p>\tfrac{1}{m+1}$.
\end{inequality}

With these specific definitions, we can now proceed to formally verify the two key assumptions of our general framework are satisfied by the trust region algorithm.

\begin{lemma}[Verification of Assumption~\ref{ass:small_true_successful} for Trust Region]
\label{lem:tr_small_true_successful}
Consider the trust region algorithm (\Cref{alg:trust_region}). Under \Cref{ass:smooth_bounded}, for any iteration $k < T_\epsilon$ that is true and small, the iteration is also successful. That is, $\rho_k \geq \eta_1$ and $\|g_k\| \geq \eta_2 \alpha_k$.
\end{lemma}
\begin{proof}

{
First, we show that $\|g_k\| \geq \eta_2 \alpha_k$. Since iteration $k$ is true and $k < T_\epsilon$, we have $\|\nabla \phi(x_k)\| > \epsilon$ and $\|\nabla \phi(x_k) - g_k\| \leq \epsilon_g + \kappa \alpha_k$. By triangle inequality, we have
\begin{align*}
\norm{g_k} &\geq \norm{\nabla \phi(x_k)} - (\epsilon_g + \kappa\alpha_k) > \epsilon - \epsilon_g - \kappa\alpha_k \\
&\geq (1-\eta)\epsilon - \kappa\alpha_k  \tag{by \Cref{ineq:eps_lowerbound_tr}} \\
&\geq \eta_2 \alpha_k \tag{since $\alpha_k \leq \bar\alpha$}.
\end{align*}
Next, we show $\rho_k \geq \eta_1$. Using Lemma 4.1 in \cite{cao2024first}, we have that $\rho_k \geq \eta_1$ if the iteration is true and the following condition on $\alpha_k$ is met:
$$
\alpha_k \leq \frac{(1-\eta_1)\kappa_{\mathrm{fcd}} (\epsilon-\epsilon_g - \kappa\alpha_k) - 2 \epsilon_g}{L + \kappa_{\mathrm{H}} + 2 \kappa}.
$$
Using $\epsilon_g < \eta\epsilon$ from \Cref{ineq:eps_lowerbound_tr}, this inequality is satisfied if
$$
\alpha_k \leq \frac{\left((1-\eta_1)\kappa_{\mathrm{fcd}}(1-\eta)-2\eta\right)\epsilon}{L+\kappa_{\mathrm{H}}+2 \kappa_{\mathrm{}}+\left(1-\eta_1\right) \kappa_{\mathrm{fcd}} \kappa}.
$$
This inequality holds by the definition of $\bar\alpha$. Thus, both conditions for a successful iteration are met.
}
\end{proof}

\begin{lemma}[Verification of Assumption~\ref{ass:large_step_makes_progress} for Trust Region]
\label{lem:tr_large_successful_progress}
Consider the trust region algorithm (\Cref{alg:trust_region}). Under \Cref{ass:smooth_bounded}, if an iteration $k < T_\epsilon$ is large and successful, then the function value decrease satisfies
$$\phi(x_k) - \phi(x_{k+1}) \geq h(\epsilon) - (2 \epsilon_f + e_k + e_k^{+}).$$
\end{lemma}
\begin{proof}
As established in Lemma 4.3 of \cite{cao2024first}, any successful iteration provides a function value decrease of at least
$$ \phi(x_k) - \phi(x_{k+1}) \ge C_{\text{prog}} \alpha_k^2 - (2\epsilon_f + e_k + e_k^+). $$
Since the iteration is large, by definition we have $\alpha_k \ge \frac{1}{\gamma_{\text{inc}}}\bar\alpha$. Therefore, 
$$ C_{\text{prog}} \alpha_k^2 \ge C_{\text{prog}} \left(\frac{\bar\alpha}{\gamma_{\text{inc}}}\right)^2. $$
By definition, $h(\epsilon) = C_{\text{prog}} \left(\frac{\bar\alpha}{\gamma_{\text{inc}}}\right)^2$. Thus, the function value decrease is at least $h(\epsilon) - (2\epsilon_f + e_k + e_k^+)$, which verifies the assumption.
\end{proof}

With both key assumptions verified, we can now present the main iteration complexity result for the trust-region method by specializing our general theorems. The high-probability iteration complexity is characterized by specializing Theorems~\ref{thm:hp_iter_bound_szo} and \ref{thm:hp_iter_bound_czo} with the progress measure $h(\epsilon)$ defined for the trust-region method.

\begin{theorem}[High-Probability Iteration Complexity for Trust Region]
\label{thm:hp_iter_bound_tr}
Consider the trust region algorithm (\Cref{alg:trust_region}). Suppose \Cref{ass:smooth_bounded} holds, the target accuracy $\epsilon$ satisfies \Cref{ineq:eps_lowerbound_tr}, and $p > p_m$.
Let $s > 0$ and $\hat{p}\in\big(p_m + \frac{s}{h(\epsilon)}, p\big)$. For any iteration count $t$ satisfying
$$ t > \frac{R}{\hat{p} - p_m - s/h(\epsilon)}, $$
the stopping time $T_\epsilon$ satisfies
$$ \P(T_\epsilon \le t) \ge 1 - \exp\left(-\frac{(p-\hat{p})^2}{2p^2}t\right) - \delta_t(s). $$
Here, $h(\epsilon)$ is as defined in \Cref{eq:h_epsilon_tr}, $R = mZ_0/h(\epsilon)+d$ and $d=\frac{m}{m+1}\left\lceil\frac{\ln \bar\alpha-\ln \alpha_0}{\ln \gamma_{d e c}}\right\rceil$. The threshold $p_m$ and the tail probability $\delta_t(s)$ depend on the zeroth-order oracle:
\begin{itemize}
    \item \textbf{For SZO Oracle:}
        $p_m = p_m^{\rm SZO} = \frac{1}{m+1}+\frac{4 m \epsilon_f}{h(\epsilon)}$, and the tail probability is
        $$
        \delta_t(s) \;=\; \exp\!\left(-\,\frac{s^2}{2 (q+2)^2 e^q \zeta_2}\,t\right)\;+\;\frac{\left(1+\tfrac{2}{q}\right)^q 2^q \zeta_q}{s^q\,t^{\,q-1}}
        $$
        for bounded $q$-th centered moment with $q \geq 2$.
    \item \textbf{For CZO Oracle:}
        $p_m = p_m^{\rm CZO} = \frac{1}{m+1}+\frac{m(4\epsilon_f+2\delta_0\epsilon_c)}{h(\epsilon)}$, and the tail probability is $\delta_t(s) = \exp\left(-\frac{s^2}{2(\epsilon_f+\epsilon_c)^2}t\right)$.
\end{itemize}
Since $h(\epsilon)=O(\epsilon^2)$, if the zeroth-order oracle is $SZO$ with bounded $q$-th moment with $q>2$, or is CZO, then the expected complexity is $\E[T_\epsilon] = O(\epsilon^{-2})$.
\end{theorem}
Theorem~\ref{thm:hp_iter_bound_tr} establishes a high-probability bound on the iteration complexity for the trust region method. This result not only provides insights into the expected complexity but also characterizes the tail behavior of the random variable of the stopping time $T_\epsilon$.

\subsection{Line Search}
\label{subsec:line_search}

We now turn our attention to the line search algorithm (\Cref{alg:line_search}). The small-step threshold is defined as $\bar{\alpha}=\min \left\{\frac{1-\theta}{0.5 L+\kappa}, \frac{2(1-2 \eta-\theta(1-\eta))}{L(1-\eta)}\right\}$, for some $\eta \in\left(0, \frac{1-\theta}{2-\theta}\right)$. The progress measure $h(\epsilon)$ is given by:
\begin{equation}
\label{eq:h_epsilon_ls}
h(\epsilon) = \frac{\theta \bar\alpha}{\gamma_{\text{inc}}\left(\max\left\{\frac{1}{\eta}, 1+\tau\right\}\right)^2} \epsilon^2.
\end{equation}

In this algorithm, the achievable accuracy, denoted by $\epsilon$, is intrinsically linked to a rejection threshold, $\epsilon_{\mathrm{rej}}$, which serves as the effective accuracy target. However, similar to the trust region setting, $\epsilon_{\mathrm{rej}}$ cannot be set to an arbitrarily small value due to the inherent biases in the zeroth- and first-order oracles ($\epsilon_f$ and $\epsilon_g$). For the convergence analysis to hold, $\epsilon_{\mathrm{rej}}$ must be sufficiently large with respect to the level of bias in the zeroth- and first-order oracles.
The relationship between the oracle biases, $\epsilon_{\mathrm{rej}}$, and the final stationarity guarantee $\epsilon$ is formalized below.

\begin{inequality}[Lower bound on $\epsilon_{\mathrm{rej}}$ and resulting accuracy]\label{ineq:eps_lowerbound_ls}
	\phantom{a}
	\begin{itemize}
    \item \textbf{For SZO Oracle:}
        \[
        \epsilon_{\mathrm{rej}} \geq \max\left\{\epsilon_g, \sqrt{\frac{4m\gamma_{\text{inc}}\epsilon_f}
        {\theta \bar{\alpha}( p-\frac{1}{m+1})}}\right\}.
        \]
    \item \textbf{For CZO Oracle:}
        \[
        \epsilon_{\mathrm{rej}} \geq \max\left\{\epsilon_g, \sqrt{\frac{2m\gamma_{\text{inc}} (2\epsilon_f+ \delta_0 \epsilon_c)}{\theta \bar{\alpha}\left(p-\frac{1}{m+1}\right)}}\right\}.
        \]
\end{itemize}
Given such a $\epsilon_{\mathrm{rej}}$, the algorithm is able to find an $\epsilon$-stationary point, where the achievable accuracy $\epsilon$ is defined as:
\[
\epsilon = \max\left\{\frac{1}{\eta}, \, 1+\tau\right\}\epsilon_{\mathrm{rej}}.
\]
\end{inequality}

\begin{lemma}[Verification of Assumption~\ref{ass:small_true_successful} for Line Search]
\label{lem:main_bound0}
Consider the line search algorithm (\Cref{alg:line_search}). Under \Cref{ass:smooth_bounded}, for any iteration $k < T_\epsilon$ that is true and small, the iteration is also successful.
\end{lemma}
\begin{proof}
	An iteration is successful if it satisfies the sufficient decrease condition and is not rejected by the criterion $\norm{g_k} \geq \epsilon_{\mathrm{rej}}$. As established in \cite{jin2024high}, any true and small step satisfies the sufficient decrease condition. We thus focus on showing that the  criterion of $\norm{g_k} \geq \epsilon_{\mathrm{rej}}$ is also met.

	Specifically, we need to show that for any true iteration $k < T_\epsilon$, we have $\norm{g_k} \geq \epsilon_{\mathrm{rej}}$. The derivation is as follows:
	\begin{align*}
		\norm{g_k} 
		&\geq \min\left\{\frac{1}{1+\tau}, 1 - \eta\right\} \norm{\nabla f(x_k)} &\text{(by Proposition 1 (iii) in \cite{jin2024high}, which uses $\epsilon \geq \frac{\epsilon_{\mathrm{rej}}}{\eta} \geq \frac{\epsilon_g}{\eta}$)}\\
		&\geq \min\left\{\frac{1}{1+\tau}, 1 - \eta\right\} \epsilon &\text{(since $k < T_\epsilon$)}. \\
		&\geq \epsilon_{\mathrm{rej}} &\text{(by the lower bound on $\epsilon$ in \Cref{ineq:eps_lowerbound_ls})}
	\end{align*}
	Here, the last inequality holds because $\eta < \frac{1-\theta}{2-\theta} < \frac12$. This implies that $\frac{1}{\eta} > \frac{1}{1-\eta}$, which implies that $\epsilon= \max\left\{\frac{\epsilon_{\mathrm{rej}}}{\eta}, \, (1+\tau)\epsilon_{\mathrm{rej}}\right\}\geq 
	\max\left\{\frac{\epsilon_{\mathrm{rej}}}{1-\eta}, \, (1+\tau)\epsilon_{\mathrm{rej}}\right\}$.

\end{proof}

\begin{lemma}[Verification of Assumption~\ref{ass:large_step_makes_progress} for Line Search]
\label{lem:large_step_makes_progress_ls}
Consider the line search algorithm (\Cref{alg:line_search}). Under \Cref{ass:smooth_bounded}, if an iteration $k < T_\epsilon$ is large and successful, then the function value decrease satisfies
\[ \phi(x_k) - \phi(x_{k+1}) \geq h(\epsilon) - (2 \epsilon_f + e_k + e_k^{+}). \]
\end{lemma}
\begin{proof}
For a successful iteration, the sufficient decrease condition implies a lower bound on the objective function decrease of
\[ \phi(x_k) - \phi(x_{k+1}) \ge \theta \alpha_k \|g_k\|^2 - (2\epsilon_f + e_k + e_k^+). \]
By definition, a successful iteration requires $\|g_k\| \ge \epsilon_{\mathrm{rej}}$. For a large iteration, the step size is bounded below by $\alpha_k \ge \frac{\bar\alpha}{\gamma_{\mathrm{inc}}}$. Substituting these bounds into the inequality gives:
\[ \phi(x_k) - \phi(x_{k+1}) \ge \theta \left(\frac{\bar\alpha}{\gamma_{\mathrm{inc}}}\right) \epsilon_{\mathrm{rej}}^2 - (2\epsilon_f + e_k + e_k^+). \]
From the relationship between $\epsilon$ and $\epsilon_{\mathrm{rej}}$ established in \Cref{ineq:eps_lowerbound_ls}, we know that $\epsilon_{\mathrm{rej}} = \epsilon / \max\{\frac{1}{\eta}, 1+\tau\}$. By substituting this into the expression for the progress, we recover the definition of $h(\epsilon)$:
\[ \theta \left(\frac{\bar\alpha}{\gamma_{\mathrm{inc}}}\right) \epsilon_{\mathrm{rej}}^2 = \theta \frac{\bar\alpha}{\gamma_{\mathrm{inc}}} \left(\frac{\epsilon}{\max\{\frac{1}{\eta}, 1+\tau\}}\right)^2 = h(\epsilon), \]
which completes the proof.
\end{proof}

With these pieces in place, we can now state the main high-probability iteration complexity result for the line search method.

\begin{theorem}[High-Probability Iteration Complexity for Line Search]
	\label{thm:hp_iter_bound_ls}
	Consider the line search algorithm (\Cref{alg:line_search}) under \Cref{ass:smooth_bounded}, with the $\epsilon_{\mathrm{rej}}$ parameter in \Cref{alg:line_search} satisfying \Cref{ineq:eps_lowerbound_ls}. Let $p > p_m$, $s > 0$ and $\hat{p}\in\big(p_m + \frac{s}{h(\epsilon)}, p\big)$. For any iteration count $t$ satisfying
	$$ t > \frac{R}{\hat{p} - p_m - s/h(\epsilon)}, $$
	the algorithm finds an $\epsilon$-stationary point with probability
	$$ \P(T_\epsilon \le t) \ge 1 - \exp\left(-\frac{(p-\hat{p})^2}{2p^2}t\right) - \delta_t(s). $$
	Here, $h(\epsilon)$ is as defined in \Cref{eq:h_epsilon_ls}, $R = mZ_0/h(\epsilon)+d$ and $d=\frac{m}{m+1}\left\lceil\frac{\ln \bar\alpha-\ln \alpha_0}{\ln \gamma_{d e c}}\right\rceil$. The threshold $p_m$ and the tail probability $\delta_t(s)$ depend on the zeroth-order oracle:
	\begin{itemize}
		\item \textbf{For SZO Oracle:}
			$p_m = p_m^{\rm SZO} = \frac{1}{m+1}+\frac{4 m \epsilon_f}{h(\epsilon)}$, and the tail probability is
   $$
   \delta_t(s) \;=\; \exp\!\left(-\,\frac{s^2}{2 (q+2)^2 e^q \zeta_2}\,t\right)\;+\;\frac{\left(1+\tfrac{2}{q}\right)^q 2^q \zeta_q}{s^q\,t^{\,q-1}}
   $$
   for bounded $q$-th centered moment with $q \geq 2$.
		\item \textbf{For CZO Oracle:}
			$p_m = p_m^{\rm CZO} = \frac{1}{m+1}+\frac{m(4\epsilon_f+2\delta_0\epsilon_c)}{h(\epsilon)}$, and the tail probability is $\delta_t(s) = \exp\left(-\frac{s^2}{2(\epsilon_f+\epsilon_c)^2}t\right)$.
	\end{itemize}
	Since $h(\epsilon)=O(\epsilon^2)$, if the zeroth-order oracle is $SZO$ with bounded $q$-th moment with $q>2$, or is CZO, then the expected complexity is $\E[T_\epsilon] = O(\epsilon^{-2})$.
	\end{theorem}
	Theorem~\ref{thm:hp_iter_bound_ls} establishes a high-probability bound on the iteration complexity for the line search method. This result not only provides insights into the expected complexity but also characterizes the tail behavior of the stopping time $T_\epsilon$. Notably, in the special case where $p > 1/2$ and the zeroth-order oracle noise is sub-exponential (a specific instance of our SZO oracle), our result recovers the high-probability complexity bound in \cite{jin2024high}.

\section{Concluding Remarks}

A practical consideration when implementing the algorithms in our framework is the choice of the step-size update parameters, $\gamma_{\mathrm{inc}}$ and $\gamma_{\mathrm{dec}}$. Theoretically, these parameters must be set to ensure an ``upward drift" of the stochastic process of the step size, which requires $p > \frac{\ln(\gamma_{\mathrm{dec}})}{\ln(\gamma_{\mathrm{dec}}/\gamma_{\mathrm{inc}})}$, where $p$ is the probability of a true iteration. This can be achieved in practice by estimating a lower bound for $p$ and selecting the parameters accordingly. 

It is also important to highlight a key implication of our results in the context of Expected Risk Minimization (ERM). In this setting, our objective function, $\phi(x)$, represents the true underlying risk, which is free from data anomalies. Our theoretical analysis is conducted with respect to this true function, while all sources of data corruption, noise, and error are modeled as inaccuracies in the oracle outputs. Therefore, by providing convergence guarantees to the true risk $\phi(x)$, our results also demonstrate generalization properties of the algorithm. Moreover, the convergence neighborhood it can achieve is cleanly dictated by the level of bias and noise in the oracles.

\bibliographystyle{alpha}
\bibliography{references}

@article{berahas2021global,
  author      = {Albert S. Berahas and Liyuan Cao and Katya Scheinberg},
  title       = {Global Convergence Rate Analysis of a Generic Line Search Algorithm with Noise},
  journal = {SIAM Journal on Optimization},
  volume = {31},
  number      = {2},
  year        = {2021},
  pages = {1489--1518},
  doi = {10.1137/19M1291832},
  aurl         = {https://arxiv.org/abs/1910.04055},
}

@article{CS17, title={Global convergence rate analysis of unconstrained optimization methods based on probabilistic models}, volume={169}, DOI={10.1007/s10107-017-1137-4}, number={2}, journal={Mathematical Programming}, author={Cartis, C. and Scheinberg, K.}, year={2017}, pages={337--375}}

@article{blanchet2019convergence,
   title={Convergence rate analysis of a stochastic trust-region method via supermartingales},
   author={Blanchet, Jose and Cartis, Coralia and Menickelly, Matt and Scheinberg, Katya},
   journal={INFORMS Journal on Optimization},
   volume={1},
   number={2},
   pages={92--119},
   year={2019},
   publisher={INFORMS}
 }

@inproceedings{bellavia2020stochastic,
	title={A Stochastic Cubic Regularisation Method with Inexact Function Evaluations and Random Derivatives for Finite Sum Minimisation},
	author={Bellavia, Stefania and Gurioli, Gianmarco and Morini, Benedetta and Toint, Philippe L},
	booktitle={Thirty-seventh International Conference on Machine Learning: ICML2020},
	year={2020}
}

@article{bellavia2022adaptive,
	title={Adaptive Regularization for Nonconvex Optimization Using Inexact Function Values and Randomly Perturbed Derivatives},
	author={Bellavia, Stefania and Gurioli, Gianmarco and Morini, Benedetta and Toint, Ph L},
	journal={Journal of Complexity},
	volume={68},
	pages={101591},
	year={2022},
	publisher={Elsevier}
}

@article{paquette2018stochastic,
   title={A stochastic line search method with expected complexity analysis},
   author={Paquette, Courtney and Scheinberg, Katya},
   journal={SIAM Journal on Optimization},
   volume={30},
   number={1},
   pages={349--376},
   year={2020},
   publisher={SIAM}
 }

@book{vershynin2018high,
   title={High-dimensional probability: An introduction with applications in data science},
   author={Vershynin, Roman},
   volume={47},
   year={2018},
   publisher={Cambridge university press}
 }

@article{cao2024first,
  title={First-and second-order high probability complexity bounds for trust-region methods with noisy oracles},
  author={Cao, Liyuan and Berahas, Albert S and Scheinberg, Katya},
  journal={Mathematical Programming},
  volume={207},
  number={1},
  pages={55--106},
  year={2024},
  publisher={Springer}
}

@article{gratton2018complexity,
  title={Complexity and global rates of trust-region methods based on probabilistic models},
  author={Gratton, Serge and Royer, Cl{\'e}ment W and Vicente, Lu{\'\i}s N and Zhang, Zaikun},
  journal={IMA Journal of Numerical Analysis},
  volume={38},
  number={3},
  pages={1579--1597},
  year={2018},
  doi = {10.1093/imanum/drx043},
}

@article{bandeira2014convergence,
  title={Convergence of trust-region methods based on probabilistic models},
  author={Bandeira, Afonso S and Scheinberg, Katya and Vicente, Luis Nunes},
  journal={SIAM Journal on Optimization},
  volume={24},
  number={3},
  pages={1238--1264},
  year={2014},
  publisher={SIAM}
}

@article{jin2024high,
	title={High probability complexity bounds for adaptive step search based on stochastic oracles},
	author={Jin, Billy and Scheinberg, Katya and Xie, Miaolan},
	journal={SIAM Journal on Optimization},
	volume={34},
	number={3},
	pages={2411--2439},
	year={2024},
	publisher={SIAM}
}

@article{berahas2023sequential,
	title={A sequential quadratic programming method with high-probability complexity bounds for nonlinear equality-constrained stochastic optimization},
	author={Berahas, Albert S and Xie, Miaolan and Zhou, Baoyu},
	journal={SIAM Journal on Optimization},
	volume={35},
	number={1},
	pages={240--269},
	year={2025},
	publisher={SIAM}
}

@article{jin2025sample,
	title={Sample complexity analysis for adaptive optimization algorithms with stochastic oracles},
	author={Jin, Billy and Scheinberg, Katya and Xie, Miaolan},
	journal={Mathematical Programming},
	volume={209},
	number={1},
	pages={651--679},
	year={2025},
	publisher={Springer}
}

@inproceedings{scheinberg2023stochastic,
	title={Stochastic adaptive regularization method with cubics: A high probability complexity bound},
	author={Scheinberg, Katya and Xie, Miaolan},
	booktitle={2023 Winter Simulation Conference (WSC)},
	pages={3520--3531},
	year={2023},
	organization={IEEE}
}

@article{menickelly2023stochastic,
	title={A Stochastic Quasi-Newton Method in the Absence of Common Random Numbers},
	author={Menickelly, Matt and Wild, Stefan M and Xie, Miaolan},
	journal={arXiv preprint arXiv:2302.09128},
	year={2023}
}

@article{JMAA2016,
	title={Deviation inequalities for martingales with applications},
	author={Fan, Xiequan and Grama, Ion and Liu, Quansheng},
	journal={Journal of Mathematical Analysis and Applications},
	volume={448},
	number={1},
	pages={538--566},
	year={2017},
	publisher={Elsevier},
	doi={10.1016/j.jmaa.2016.10.051},
	url={https://www.sciencedirect.com/science/article/pii/S0022247X16307065}
}

@article{FukNagaev1971,
	title={Probability inequalities for sums of independent random variables},
	author={Fuk, D. H. and Nagaev, S. V.},
	journal={Theory of Probability \& Its Applications},
	volume={16},
	number={4},
	pages={643--660},
	year={1971},
	publisher={SIAM}
}

@book{Petrov1995,
	title={Limit Theorems of Probability Theory: Sequences of Independent Random Variables},
	author={Petrov, Valentin V.},
	year={1995},
	publisher={Oxford University Press},
	address={Oxford},
	series={Oxford Studies in Probability}
}
 
\appendix

\end{document}